\documentclass{amsart}
\usepackage[utf8]{inputenc}

\usepackage{amsmath}
\usepackage{amsthm}
\usepackage{amssymb}
\usepackage{amsfonts}
\usepackage{amsopn}
\usepackage{amsrefs}
\usepackage{hyperref}
\usepackage[color=green!30]{todonotes}

\usepackage{import}

\usepackage{tikz}
\usetikzlibrary{intersections,decorations.markings,patterns,decorations.pathreplacing,matrix,arrows}

\def\Z{\mathbb{Z}}
\def\Q{\mathbb{Q}}
\def\R{\mathbb{R}}

\def\C{\mathbb{C}}
\def\wt#1{\widetilde{#1}}
\def\ol#1{\overline{#1}}

\numberwithin{equation}{section}
\newtheorem{theorem}[equation]{Theorem}
\newtheorem{proposition}[equation]{Proposition}
\newtheorem{lemma}[equation]{Lemma}
\newtheorem{corollary}[equation]{Corollary}
\theoremstyle{definition}
\newtheorem{definition}[equation]{Definition}

\theoremstyle{remark}
\newtheorem{remark}[equation]{Remark}

\newtheorem*{ack}{Acknowledgments}

\title{Untwisting number and Blanchfield pairings}

\author{Maciej Borodzik}
\address{Institute of Mathematics, Polish Academy of Science, ul. \'Sniadeckich 8,
00-656 Warsaw, Poland}
\address{Institute of Mathematics, University of Warsaw, ul. Banacha 2,
02-097 Warsaw, Poland}
\email{mcboro@mimuw.edu.pl}

\begin{document}
\begin{abstract}
In this note we use Blanchfield forms to study knots that can be turned into an unknot using a single $\ol{t}_{2k}$ move.
\end{abstract}
\maketitle

\section{Overview}
Let $K\subset S^3$ be a knot and $k\in\Z\setminus\{0\}$.
In this paper by a \emph{$k$--twisting move} we mean a move depicted in Figure~\ref{fig:k_untwist},
that is, a full right $k$--twist on two strands of $K$ going in the opposite direction (in \cite{Przy} this move
is called a $\ol{t}_{2k}$--move). We will call a knot \emph{$k$--simple}
if it can be unknotted by a single $k$--untwisting move. A knot is \emph{algebraically $k$--simple} if a single $k$--untwisting move
turns it into a knot with Alexander polynomial $1$.

\begin{figure}[h]
\begin{tikzpicture}[y=2pt,x=2pt,yscale=-1, inner sep=0pt, outer sep=0pt]
\begin{scope}[thick,decoration={
    markings,
    mark=at position 0.5 with {\arrow{>}}}
    ] 
  \draw[postaction={decorate}]
    (26.8363,96.8616) .. controls (49.0963,93.9030) and (56.4631,107.2535) ..
    (75.4062,100.0744);
  \draw[postaction={decorate}]

    (81.0759,97.9955) .. controls (96.5149,89.4987) and (107.3195,106.9761) ..
    (120.9524,101.2083);
  \draw[postaction={decorate}]

    (128.7009,97.6176) .. controls (131.7275,96.1448) and (134.8298,95.4826) ..
    (138.1503,94.9717);
\begin{scope}[xshift=-180,yshift=-80]
  \draw[postaction={decorate}]

    (120,95) -- 
    (230,95);
  \draw[postaction={decorate}]

    (230,100) -- 
    (120,100);
\end{scope}
\draw[thick, dashed,->] (88,65) -- ++ (0,25) node [midway, right,scale=0.8] {\ twisting move};
\end{scope}
\begin{scope}[thick,decoration={markings,mark=at position 0.5 with {\arrow{<}}}]
  \draw[postaction={decorate}]

    (25.7024,101.9643) .. controls (26.1564,101.9151) and (26.5980,101.8230) ..
    (27.0303,101.8421) .. controls (37.5839,102.3070) and (40.1972,103.7781) ..
    (47.8140,101.0193);
  \draw[postaction={decorate}]

    (52.1607,98.1845) .. controls (69.0885,89.8972) and (85.1080,106.8904) ..
    (100.7306,100.6414);
  \draw[postaction={decorate}]

    (104.5104,98.7515) .. controls (117.1519,91.0529) and (127.7560,102.8947) ..
    (138.7173,101.2083);
\end{scope}
\end{tikzpicture}
\caption{A $k$--twisting move for $k=2$. Note that the strands in the picture go in different directions.}\label{fig:k_untwist}
\end{figure}
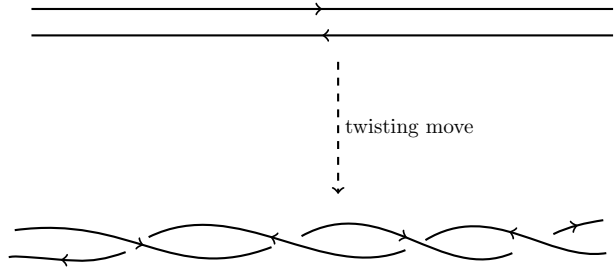

Our first result gives an obstruction to the untwisting move in terms of the algebraic unknotting number \cite{Fo93,Muk90,Sae99}.
\begin{theorem}\label{cor:unknotting}
Suppose $K$ is an algebraically $k$--simple knot.
If $k$ is odd, then $K$
can be turned into a knot with Alexander polynomial $1$ using at most two crossing changes. If $k$ is even, then at most three crossing
changes are enough to turn $K$ into a knot with Alexander polynomial $1$.
\end{theorem}
Our second result restricts the homology of the double branched cover of an algebraically $k$--simple knot.
\begin{theorem}\label{thm:iscyclic}
Suppose $K$ is an algebraically $k$--simple knot. Denote by $\Sigma(K)$ the double branched cover of $K$. Then $H_1(\Sigma(K);\Z)$ is cyclic.
\end{theorem}
Both Theorem~\ref{cor:unknotting} and Theorem~\ref{thm:iscyclic} follow from the following result, which is the main technical result of this paper.
\begin{theorem}\label{thm:main}
Suppose $K$ is an algebraically $k$--simple knot. Then there exists a polynomial $\alpha(t)\in\Z[t,t^{-1}]$ satisfying $\alpha(1)=0$, 
$\alpha(t^{-1})=\alpha(t)$, such that
the matrix
\[\begin{pmatrix} \alpha(t) & 1 \\ 1 & -k \end{pmatrix}\]
represents the Blanchfield pairing for $K$.
\end{theorem}
Theorem~\ref{thm:main} can be regarded as a generalization of \cite[Theorem 3.2(b)]{Przy}.

It is possible to generalize the techniques used in this paper to study knots that are untwisted with several $\ol{t}_{2k}$ moves, possibly with varying
the twisting coefficients $k$. This generalization is straightforward, we omit to make the paper shorter and more concise.

Proof of Theorem~\ref{thm:main} is given in Section~\ref{sec:proof}.
Proof of Theorem~\ref{cor:unknotting} is given in Section~\ref{sec:applications}.
Section~\ref{sec:linkingforms} contains
the proof of a stronger version of Theorem~\ref{thm:iscyclic}.

\begin{ack}
The author is grateful to A.~Conway, S.~Friedl, C.~Livingston, W.~Politarczyk and J.~Przytycki for fruitful conversations.
He is especially indebted to A.~Ranicki for pointing out \cite[Proposition 3.30]{Ran02}.
The research is supported by
the National Science Center grant 2016/22/E/ST1/00040.
\end{ack}

\section{Blanchfield pairing}\label{sec:blanchfield}
Let $K\subset S^3$ be a knot and let $M_K$ denote its zero-framed surgery. Denote by $\wt{M}_K$ the universal abelian cover of $M_K$. The chain
complex $C_*(\wt{M}_K;\Z)$ admits the action of the deck transform and thus it has a structure of a $\Lambda$--module, where $\Lambda=\Z[t,t^{-1}]$.
The homology of this complex, regarded 
as a $\Lambda$--module, is denoted by $H_*(M_K;\Lambda)$. The module $H_1(M_K;\Lambda)$ is called the \emph{Alexander module} of the knot $K$.
\begin{remark}
Usually the Alexander module is defined using knot complements instead of zero--framed surgeries, but the two definitions are
equivalent; see e.g. \cite{FP}.
\end{remark}
The ring $\Lambda$ has a naturally defined convolution $t\mapsto t^{-1}$.
The Blanchfield pairing defined in
\cite{Bl57} for $K$ is a sesquilinear symmetric pairing $H_1(M_K;\Lambda)\times H_1(M_K;\Lambda)\to Q/\Lambda$, where $Q$ is the field
of fractions for $\Lambda$. We refer to \cite{FP,Hi12} for a precise and detailed construction of the Blanchfield pairing and \cite{Con,CFT} for
generalizations.

\begin{definition}\label{def:repres}
We say that an $n\times n$ matrix $A$ with entries in $\Lambda$ \emph{represents} the Blanchfield pairing if $H_1(M_K;\Lambda)\cong \Lambda^n/A\Lambda^n$
as a $\Lambda$--module, under this identification the Blanchfield pairing has form $(a,b)\mapsto a^TA^{-1}\ol{b}$ and moreover
$A(1)$ is diagonalizable over $\Z$.
\end{definition}
It is known, see \cite{Kear}, that every Blanchfield pairing can be represented by a finite matrix. The minimal size of a matrix representing the Blanchfield
pairing of a knot is denoted by $n(K)$. It is equal to the algebraic unknotting number $u_a(K)$; see \cite{BF3,BF1}.

The invariant $n(K)$ can also be generalized for other coefficient ring $R$. In this paper we restrict to rings $R$ that are subrings of $\C$.
We denote by $n_R(K)$ the minimal size of a matrix over $R[t,t^{-1}]$ representing the
Blanchfield pairing over $R[t,t^{-1}]$.
We have that $n_R{K}\le n_{R'}(K)$
if $R'$ is a subring of $R$. 
Often $n_R(K)$ is easier to compute than $n(K)=n_\Z(K)$,
for example the value of $n_{\R}$ can be calculated from the Tristram--Levine signature \cite{BF2}. One motivation of this paper is to give
a geometric interpretation of $n_R(K)$ for some rings $R$.

\section{Proof of Theorem~\ref{thm:main}}\label{sec:proof}
The main ingredient in the proof of Theorem~\ref{thm:main} is the following.
\begin{theorem}[see \expandafter{\cite[Theorem 2.6]{BF1}}]\label{thm:26}
Suppose $W_K$ is a topological four--manifold such that $\partial W_K=M_K$, $\pi_1(W_K)=\Z$ and the inclusion induced map $H_1(M_K;\Z)\to H_1(W_K;\Z)$
is an isomorphism. Then $H_2(W_K;\Lambda)$ is free of rank $b_2(W_K)$. Moreover if $A$ is matrix over $\Lambda$
representing the twisted intersection form on $H_2(W_K;\Lambda)$ in some basis of $H_2(W_K;\Lambda)$, then $A$ also represents
the Blanchfield pairing on $M_K$.
\end{theorem}

In the light of Theorem~\ref{thm:26}, the
proof of Theorem~\ref{thm:main} consists of constructing an appropriate manifold $W_K$ and applying Theorem~\ref{thm:26}.
The construction begins with noticing that the
twisting move can be realized by a surgery. Namely we have the following well-known fact.
\begin{proposition}\label{prop:surgerytwist}
A $k$--twisting move can be realized by a $-1/k$ surgery on a knot. That is, if $K_2$ arises from $K_1$ by a $k$--twisting move, then there is a simple closed
circle $C$ disjoint from $K_1$, such that $C$ bounds a smooth disk intersecting $K_2$ at two points with opposite signs and
such that the $-1/k$ surgery on $C$ transforms $K_1$ into $K_2$; see Figure~\ref{fig:twist_surg}
\end{proposition}
\begin{figure}
\begin{tikzpicture}[y=2pt,x=2pt,yscale=-1, inner sep=0pt, outer sep=0pt]
  \draw
    (26.8363,96.8616) .. controls (49.0963,93.9030) and (56.4631,107.2535) ..
    (75.4062,100.0744);
  \draw
    (25.7024,101.9643) .. controls (26.1564,101.9151) and (26.5980,101.8230) ..
    (27.0303,101.8421) .. controls (37.5839,102.3070) and (40.1972,103.7781) ..
    (47.8140,101.0193);
  \draw
    (52.1607,98.1845) .. controls (69.0885,89.8972) and (85.1080,106.8904) ..
    (100.7306,100.6414);
  \draw
    (81.0759,97.9955) .. controls (96.5149,89.4987) and (107.3195,106.9761) ..
    (120.9524,101.2083);
  \draw
    (104.5104,98.7515) .. controls (117.1519,91.0529) and (127.7560,102.8947) ..
    (138.7173,101.2083);
  \draw
    (128.7009,97.6176) .. controls (131.7275,96.1448) and (134.8298,95.4826) ..
    (138.1503,94.9717);
\begin{scope}[xshift=-180,yshift=-50]
  \draw
    (120,95) -- 
    (164.5,95);
  \draw
    (120,100) -- 
    (164.5,100);
  \draw
    (171,95) --
    (230,95);
  \draw
    (171,100) --
    (230,100);
  \draw
    (182.7515,92.5149) .. controls (164.0099,70.4633) and (161.9341,122.7147) ..
    node[near start,above,scale=0.7] {$1/k$} (182.1845,101.3973);
\end{scope}
\end{tikzpicture}
\caption{The $1/k$ surgery on the circle in the top picture induces $k$ full left twists of the two strands passing through the circle.}\label{fig:twist_surg}
\end{figure}
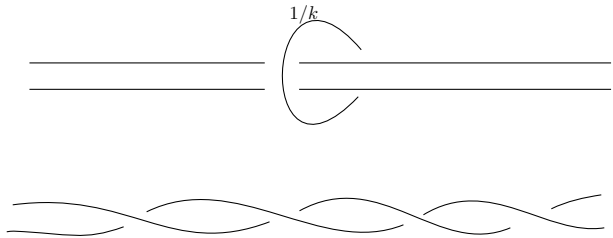
\begin{remark}
The move described in Figure~\ref{fig:twist_surg} is a special case of the Rolfsen twist, see \cite[Figure~5.27]{Stipsicz}.
It can be seen on \cite[Figure 3.12]{Sav} that the surgery with a positive coefficient (i.e. the $1/k$ surgery if $k>0$)
gives rise to a left $k$--twist and the surgery with a
negative coefficient (i.e. the $-1/k$ surgery with $k>0$) gives rise to a right $k$--twist.
\end{remark}
The surgery in Figure~\ref{fig:twist_surg} can be changed into a surgery with integer coefficients as in Figure~\ref{fig:othersurg}
by a `slam-dunk' operation, see \cite[Section 5.3]{Stipsicz}.
\begin{figure}
\begin{tikzpicture}
\begin{scope}[thick,decoration={
    markings,
    mark=at position 0.5 with {\arrow{>}}}
    ]
\begin{scope}[xshift=-2cm,yshift=-3cm] 
\draw[postaction={decorate}](-0.15,3) -- (-0.15,5);
\draw[postaction={decorate}](0.15,5) -- (0.15,3);
\draw(-0.25,3.6) arc [x radius=1,y radius=0.4, start angle=105, delta angle=330];
\draw(1.3,3.3) node [scale=0.8] {$1/k$};
\draw(-0.15,2.7) -- (-0.15,2);
\draw(0.15,2.7) -- (0.15,2);
\end{scope}
\begin{scope}[xshift=6cm]
\draw[postaction={decorate}](-3.15,0) -- (-3.15,2);
\draw[postaction={decorate}](-2.85,2) -- (-2.85,0);
\draw(-3.25,0.6) arc [x radius=1,y radius=0.4, start angle=105, delta angle=200] coordinate (A);
\path[draw=none] (A) arc [x radius=1, y radius=0.4, start angle=305, delta angle=50] coordinate (B);
\draw (B) arc [x radius=1, y radius=0.4, start angle=355, delta angle=80];
\draw (A) ++ (-0.1,0.3) arc [x radius=1, y radius=0.4, start angle=190, delta angle=300];
\draw (-4.2,0.3) node[scale=0.8] {$c_0$};
\draw (-0.3,0.3) node[scale=0.8] {$c_1$};
\draw(-3.15,-0.3) -- (-3.15,-1);
\draw(-2.85,-0.3) -- (-2.85,-1);
\draw(-3.4,0.8) node [scale=0.8] {$0$};
\draw (-1.7,0.8) node [scale=0.8] {$-k$};
\end{scope}
\end{scope}
\end{tikzpicture}
\caption{Changing a $1/k$ surgery on a circle to a surgery on a two-component link with framings $0$ and $-k$.}\label{fig:othersurg}
\end{figure}
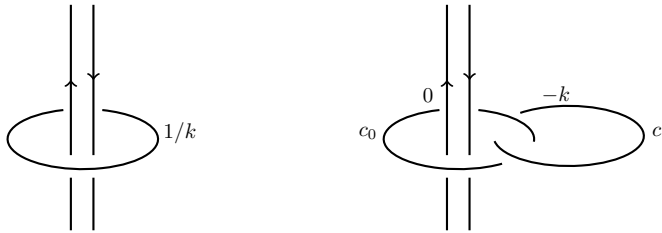

Suppose $J$ is a knot with Alexander polynomial $1$ and $K$ is a knot resulting from $J$ by applying a full left $k$--twist
(so $J$ is obtained from $K$ by a full right $k$--twist).
Let $M_J$ be the zero-surgery on $J$ and $M_K$ the zero--surgery on $K$. By \cite[Theorem 117B]{FreedmanQuinn}
$M_J$ is a boundary of a topological four--manifold that is a homotopy $D^3\times S^1$. Denote this four--manifold by $W_J$.

A full left $k$--twist on $J$ can be realized as a surgery on a two-component link with framings $0$ and $-k$ as in Figure~\ref{fig:othersurg}.
Let $c_0$ and $c_1$ denote the components of this link. The curve $c_0$ has framing $0$, $c_1$ has framing $k$. Both $c_0$ and $c_1$ are curves disjoint from $J$,
so we can and will assume that they are separated from a small neighborhood of $J$ in $S^3$. Performing a 0--surgery on $J$ does not affect these curves, therefore
$c_0$ and $c_1$ can also be viewed as curves on $M_J$. Now performing surgery on $c_0$ and $c_1$ produces $M_K$.

The trace of the surgery on $c_0$ and $c_1$ yields a cobordism between $M_J$ and $M_K$. Call this cobordism $W_{JK}$.
Define now
\[W_K=W_J\cup W_{JK}\]
so that $\partial W_K=M_K$.
We have the following fact.
\begin{lemma}\label{lem:cobounds}
We have $\pi_1(W_K)\cong\Z$, $H_1(W_K;\Z)\cong \Z$ and the inclusion of $M_K$ to $W_K$ induces an isomorphism on the first homology.
Moreover $H_2(W_K;\Z)\cong\Z^2$ and there exists spherical generators of $H_2(W_K;\Z)$.
\end{lemma}
\begin{proof}
The homology groups of $W_K$ are calculated using the Mayer-Vietoris sequence. The manifold $W_K$ is obtained
from $W_J$ by adding two-handles along null-homologous curves $c_0$ and $c_1$.
This shows that $H_1(W_K;\Z)\cong\Z$ and $H_2(W_K;\Z)\cong\Z^2$.

To compute $\pi_1$ we observe that $\pi_1(W_J)\cong\Z$. Hence $c_0,c_1$ being null-homologous
are also null-homotopic. The van Kampen theorem implies  that $\pi_1(W_K)\cong\Z$. 

To show that the generators of $H_2(W_K;\Z)$ can be chosen to be spherical we again use the fact that $c_0$ and $c_1$ are null-homotopic in $W_J$.
This implies that $c_0$ and $c_1$ bound disks $D_0$ and $D_1$ in $W_J$. The disk $D_1$
can be chosen to be the obvious disk on $M_J$, but $D_0$ is in general only an immersed disk and it cannot lie on $M_J$ (because in general $c_0$ is
not null-homotopic on $M_J$). We can form spheres $\Sigma_0$ and $\Sigma_1$
by adding to $D_0$ and $D_1$ the cores of the two-handles that are attached. It is clear that the homology classes $[\Sigma_0]$ and $[\Sigma_1]$
generate $H_2(W_K;\Z)$. Moreover, by construction, $\Sigma_1$ is a smoothly embedded sphere and $\Sigma_0$ can be chosen to intersect $\Sigma_1$
precisely at one point.

Finally, in order to prove that the inclusion induced map  $H_1(M_K;\Z)\to H_1(W_K;\Z)$ is an isomorphism, invert the cobordism $W_{JK}$, that is,
present $W_{JK}$ as $M_K\times[0,1]$ with two two--handles attached. The attaching curves of these handles are homologically trivial (but not necessarily
homotopy trivial, $\pi_1(M_K)$ can be complicated),
hence the boundary inclusion induces an isomorphism $H_1(M_K;\Z)\cong H_1(W_{JK};\Z)$. Clearly $H_1(W_{JK};\Z)\cong H_1(W_K;\Z)$.

\end{proof}
Lemma~\ref{lem:cobounds} gives us two spheres $\Sigma_0,\Sigma_1\subset W_K$,
which are the generators of $H_2(W_K;\Z)$. Choose a basepoint $x_0=\Sigma_0\cap\Sigma_1$. This choice allows us to consider $\Sigma_0$
and $\Sigma_1$ as elements of $\pi_2(W_K,x_0)$.
\begin{lemma}\label{lem:generates}
The group $\pi_2(W_K,x_0)$ is freely generated as a $\Lambda=Z[\pi_1(W_K,x_0)]$--module by classes of $\Sigma_0$ and $\Sigma_1$. In particular
$\pi_2(W_K,x_0)\cong\Lambda^2$.
\end{lemma}
\begin{proof}
The space $W_K$ is obtained from $W_J$ by attaching two two--handles along null-homotopic curves $c_0$ and $c_1$. We have that $\pi_1(W_J)=\Z$
and $\pi_2(W_J)=0$ by definition. The statement follows from \cite[Proposition 3.30]{Ran02}.
\end{proof}

We will use Lemma~\ref{lem:generates} in connection with the following well-known result.
\begin{lemma}\label{lem:lambdaiso}
We have an isomorphism of $\Lambda$--modules $\pi_2(W_K,x_0)\cong \pi_2(\wt{W}_K,\wt{x}_0)\cong H_2(\wt{W}_K;\Z)\cong H_2(W_K;\Lambda)$.
\end{lemma}
\begin{proof}

The
first isomorphism in the lemma is the isomorphism of higher homotopy groups under the covering map. The second is the Hurewicz isomorphism because $\wt{W}_K$ is
simply connected. The third isomorphism is the definition of the twisted homology groups.
\end{proof}
In particular,
Lemma~\ref{lem:generates} together with Lemma~\ref{lem:lambdaiso} gives
a simple and independent argument that $H_2(W_K;\Lambda)$ is a free $\Lambda$--module, compare \cite[Lemma 2.7]{BF1}.

\begin{corollary}\label{cor:liftsgenerate}
The (classes of the) lifts of $\Sigma_0$ and $\Sigma_1$ to $\wt{W}_K$ generate $H_2(W_K;\Lambda)$ as a $\Lambda$--module.
\end{corollary}
Recall that $A(t)$ is a matrix over $\Lambda$ representing the intersection
form on $H_2(W_K;\Lambda)$.
The following result together with Theorem~\ref{thm:26} gives the proof of Theorem~\ref{thm:main} from the introduction.
\begin{theorem}\label{thm:form}
The matrix $A(t)$ has form
\[\begin{pmatrix} \alpha(t) & 1\\ 1 & -k\end{pmatrix},\]
where $\alpha(t)\in \Lambda$ is such that $\alpha(1)=0$ and $\alpha(t^{-1})=\alpha(t)$.
\end{theorem}
\begin{proof}
By Corollary~\ref{cor:liftsgenerate} 
the entries of $A(t)$ are twisted intersection indices of $\Sigma_0$ and $\Sigma_1$.
For example, the bottom-right entry of $A(t)$
is equal to the twisted intersection index of $\Sigma_1$ and $\Sigma_1'$, where $\Sigma'_1$ is a small perturbation of $\Sigma_1$
intersecting $\Sigma_1$ in finitely many points.

To compute the twisted intersection index of $\Sigma_1$ and $\Sigma_1'$, choose a basing for $\Sigma_1$, $\Sigma_1'$,
that is a path $\gamma$ from $x_0$ to $\Sigma_1$ and a path $\gamma'$ from $x_0$ to $\Sigma_1'$. Let $x,x'$
be the end points of $\gamma$ and $\gamma'$.

For any intersection point $y\in\Sigma_1$ and $\Sigma_1'$
we choose a smooth path $\rho_y$ from $x$ to $y$ on $\Sigma_1$ and a path $\rho_y'$ from $x'$ to $y$ on $\Sigma_1'$; see Figure~\ref{fig:paths}.

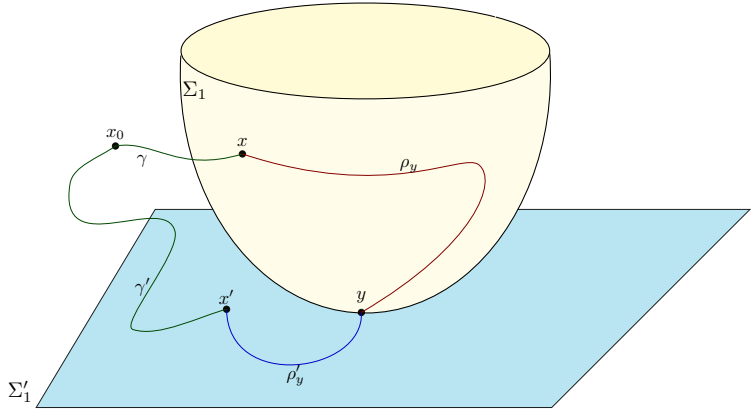
\begin{figure}
\definecolor{c0ea6d3}{RGB}{14,166,211}
\definecolor{c0c0500}{RGB}{12,5,0}
\definecolor{c0000c8}{RGB}{0,0,200}
\definecolor{c004800}{RGB}{0,72,0}
\definecolor{c840000}{RGB}{132,0,0}

\begin{tikzpicture}[y=3pt,x=3pt,yscale=-1, inner sep=0pt, outer sep=0pt]
  \path[draw=black,fill=c0ea6d3,opacity=0.886,miter limit=4.00,fill
    opacity=0.294,line width=0.020pt] (70,160) -- (85,135) --(160,135) -- (135,160) -- cycle;
  \path[draw=black,thin,fill=yellow!10]
    (134.7559,115.3322)arc(-4.951:58.193:23.340031 and
    30.144)arc(58.193:121.337:23.340031 and 30.144)arc(121.337:184.481:23.340031
    and 30.144);
  \path[draw=black,thin,fill=yellow!20]
    (127.9400,110.7281)arc(-45.000:44.901:23.245537 and
    6.048)arc(44.901:134.801:23.245537 and 6.048)arc(134.801:224.702:23.245537 and
    6.048)arc(224.702:314.602:23.245537 and 6.048);
  \path[draw=black,fill=c0c0500,opacity=0.886,miter limit=4.00,fill
    opacity=0.986,line width=0.056pt]
    (94,147.6) circle (0.4); \draw (94,146) node [scale=0.7] {$x'$};
  \path[draw=black,fill=c0c0500,opacity=0.886,miter limit=4.00,fill
    opacity=0.986,line width=0.056pt]
    (111,148) circle (0.4); \draw (111,146) node [scale=0.7] {$y$};
  \path[draw=black,fill=c0c0500,opacity=0.886,miter limit=4.00,fill
    opacity=0.986,line width=0.056pt] (96,128) circle (0.4);
\draw (96,126.5) node[scale=0.7] {$x$};
  \path[draw=black,fill=c0c0500,opacity=0.886,miter limit=4.00,fill
    opacity=0.986,line width=0.056pt]
    (80,127) circle (0.4);\draw (80,125.5) node [scale=0.7] {$x_0$};
  \path[draw=c0000c8,line join=miter,line
    width=0.056pt] (94,147.6) .. controls (94.1383,158.0702) and
    (111.4095,155.6654) .. node [midway,below,scale=0.7] {$\rho'_y$} (111,148);
  \path[draw=c004800,line join=miter,line cap=butt,even odd rule,line
    width=0.056pt] (96,128) .. controls (87.9014,130.4933) and
    (84.7752,126.1535) .. node [below=3pt,scale=0.7,near end] {$\gamma$} (80,127) .. controls (77.9760,128.4566) and
    (74.6446,129.4961) .. (74.2723,131.7299) .. controls (72.5559,142.0285) and
    (86.2278,132.6827) .. (87.5015,137.7775) .. controls (88.1933,140.5447) and
   (80.4222,149.5604) .. node[left,scale=0.7,midway] {$\gamma'$} (82.2098,150.1562) .. controls (85.4145,151.2245) and
   (90.7601,148.4302) .. (94,147.6);
  \path[draw=c840000,line join=miter,line cap=butt,even odd rule,line
    width=0.056pt] (96,128) .. controls (114.3802,134.0974) and
    (123.7400,127.7629) .. node [above,midway,scale=0.7] {$\rho_y$} (125.7876,129.3578) .. controls (128.2984,131.3134) and
    (125.4998,139.4373) .. (111,148);
  \draw (90,120) node [scale=0.8] {$\Sigma_1$};
\draw (68,158) node [scale=0.8] {$\Sigma_1'$};
\end{tikzpicture}
\caption{Notation in proof of Theorem~\ref{thm:form}. In the four-dimensional situation
the intersection of $\Sigma$ and $\Sigma'$ at $y$ is transverse.}\label{fig:paths}
\end{figure}
Let $\theta_y$ be the loop
$(\gamma')^{-1}(\rho_y')^{-1}\rho_y\gamma$. Define $n_y\in\Z$ to be the homology class of $\theta_y$ in $H_1(W_K;\Z)\cong\Z$. Finally, let
$\epsilon_y$ be the sign of the intersection point $y$ assigned in the usual way, that is, if $T_y\Sigma_1\oplus T_y\Sigma_1'=T_yW_{K}$ agrees with
 the orientation, we set $\epsilon_y=+1$, otherwise we set $\epsilon_y=-1$.

Given these definitions, the twisted intersection index of $\Sigma_1$ and $\Sigma_1'$ is equal to
\begin{equation}\label{eq:twistint}
\sum_{y\in\Sigma_1\cap\Sigma_1'} \epsilon_yt^{n_y}\in \Z[t,t^{-1}].
\end{equation}
In general this sum might depend on the choice of $\rho_y$ and $\rho_y'$. However if any smooth closed curve on $\Sigma_1$ and on $\Sigma_1'$
is homologically trivial in $W_K$ (in the language of \cite[Section 3.2]{BF3} this means that $\Sigma_1$ and $\Sigma_1'$ are homologically invisible in $W_K$),
the definition does not depend on paths $\rho_y$ and $\rho_y'$. In the present situation
$\Sigma_1$ and $\Sigma_1'$ are immersed (and even embedded) spheres, so they are homologically invisible, in particular \eqref{eq:twistint}
is a well-defined Laurent polynomial.

As $\Sigma_1$ and $\Sigma_1'$ are embedded spheres, we claim more, namely that $n_y$ does not depend on $y$. In fact, suppose $z$
is another intersection point of $\Sigma_1$ and $\Sigma_1'$.
If $n_z\neq n_y$, then
the curve $\delta=\rho_y\rho_z^{-1}\rho'_z(\rho_y')^{-1}$ is not homology trivial in $W_K$. As $\Sigma_1'$ is a perturbation of $\Sigma_1$,
the path $\rho'_z(\rho_y')^{-1}$ can be pushed by a homotopy (in $W_K$) to a path $\wt{\rho}$ on $\Sigma_1$ having the same endpoints. Then
$\rho_y\rho_z^{-1}\wt{\rho}$ is a loop hootomically equivalent to $\delta$, but this is a loop on a smoothly embedded sphere $\Sigma_1$. Hence it is
contractible in $W_K$.
This shows that $n_y=n_z$.

We conclude that the twisted intersection index of $\Sigma_1$ and $\Sigma_1'$ is equal to the standard intersection number of $\Sigma_1$ and $\Sigma_1'$
(which is equal to the self-intersection of $\Sigma_1$, that is $-k$) multiplied by $t^{n_y}$. We can choose a basing for
$\Sigma_1'$ in such a way that $n_y=0$.

An analogous, but simpler argument shows that $\Sigma_0\cdot\Sigma_1=\pm 1$. Indeed by construction $\Sigma_0\cap\Sigma_1$ consists of a single point.
It follows that the twisted intersection between $\Sigma_0$ and $\Sigma_1$ is $\pm t^m$ for some $m$. We choose a basing for $\Sigma_0$ in such a way that
$m=0$. We can also choose an orientation of $\Sigma_0$ in such a way that the sign is positive.
\end{proof}

\begin{remark}
There is an alternative calculation of the matrix $A$ using Rolfsen's argument \cite{Rol75}. However one still has to make some effort proving
that $A$ represents not only the Alexander module, but also the Blanchfield pairing.
\end{remark}

\section{Proof of Theorem~\ref{cor:unknotting}}\label{sec:applications}
We begin with proving Theorem~\ref{cor:unknotting}. The following corollary deals with the first part of this theorem.
\begin{corollary}\label{thm:unknotting}
Suppose $K$ is an algebraically $k$--simple and $k$ is odd. Then there are at most two crossing
changes that turn $K$ into a knot with Alexander polynomial $1$.
\end{corollary}
\begin{proof}
We have $A(1)=\left(\begin{smallmatrix} 0 & 1 \\ 1 &-k\end{smallmatrix}\right)$. As $k$ is odd, this matrix is diagonalizable over $\Z$.
By
\cite[Theorem 1.1]{BF3} we infer that the algebraic unknotting number of $K$ is at most $2$.
\end{proof}
If $k$ is even, then $A(1)$ is not diagonalizable over $\Z$, but $A(1)\oplus (1)$ is diagonalizable. The block matrix $A(t)\oplus (1)$ is a $3\times 3$
matrix over $\Lambda$ representing the Blanchfield pairing, so the algebraic unknotting number of $K$ is bounded from above by $3$. This shows the second
part of Theorem~\ref{cor:unknotting}.

We have the following consequence of Theorem~\ref{thm:main}.
\begin{theorem}
Suppose $K$ can be algebraically $k$--simple. Let $R_k=\Z[\frac1k]$. Then
$n_{R_k}=1$.
\end{theorem}
\begin{proof}
By Theorem~\ref{thm:main} we know that the Blanchfield pairing
over $\Z$ can be represented by a matrix of form $\left(\begin{smallmatrix} \alpha(t) & 1 \\ 1 & -k\end{smallmatrix}\right)$. The same matrix represents
the Blanchfield pairing over $R_k$, but over $R_k$ this matrix is congruent to a matrix $\left(\begin{smallmatrix} \wt{\alpha}(t) & 0 \\ 0 & 1\end{smallmatrix}\right)$
for $\wt{\alpha}(t)\in R_k[t,t^{-1}]$.
By \cite[Proposition 1.7.1]{Ran81} (see also \cite[Proposition 3.1]{BF1}) the matrix $(\wt{\alpha}(t))$ also represents the Blanchfield pairing over $R_k[t,t^{-1}]$.
\end{proof}

The following corollary is well known, see \cite{Przy}.
\begin{corollary}
If $K$ is algebraically $k$--simple, then its Alexander polynomial is equal to $\Delta_K(t)=1+k\alpha(t)$, where $\alpha(t)\in\Z[t,t^{-1}]$.
\end{corollary}
\begin{proof}
This follows from Theorem~\ref{thm:unknotting} because if $A(t)$ represents the Blanchfield pairing of a knot $K$, then $\Delta_K(t)=\det A(t)$
up to multiplication by a unit in $\Lambda$.
\end{proof}

\section{Linking forms}\label{sec:linkingforms}
An abstract \emph{linking pairing} is a pair $(H,l)$,
where $H$ is a finite abelian group of an odd order and $l$ is a bilinear symmetric pairing $l\colon H\times H\to\Q/\Z$,
As a model example, if $Y$ is a closed three--manifold with $b_1(Y)=0$, there is defined a linking pairing $l(Y)$ on $H=H_1(Y;\Z)$. If $Y=\Sigma(K)$
is the double branched cover of a knot $K$, we denote this pairing by $l(K)$. It is known that the linking pairing $l(K)$ is represented by $V+V^T$,
where $V$ is the Seifert matrix for $K$. The meaning of `represented' is explained in the following definition.
\begin{definition}
Let $P$ be an $n\times n$ matrix with integer coefficients and such that $\det P$ is odd. The \emph{linking form represented by $P$} is the pair
$(H(P),l(P))$, where $H(P)=\Z^n/P\Z^n$ and $l(P)$ is the bilinear
form defined by
\begin{align*}
\Z^n/P\Z^n\times \Z^n/P\Z^n&\to \Q/\Z\\
(a,b)&\mapsto a^T P^{-1} b\bmod 1.
\end{align*}
\end{definition}
We have the following relation between the Blanchfield form for $K$ and the linking form $l(K)$.
\begin{proposition}[see \expandafter{\cite[Lemma 3.3]{BF1}}]
If $A$ is a matrix over $\Lambda$ representing the Blanchfield pairing, then $l(A(-1))=2l(K)$.
\end{proposition}
Here $2l(K)$ means the linking pairing with the same underlying group as $l(K)$, but the linking form is multiplied by $2$; compare \cite[Section 3]{BF1}.

We can use this result to obtain the following corollary.
\begin{corollary}\label{cor:linking_form_obstruction}
Suppose $K$ is algebraically $k$--simple. Then the linking form $2l(K)$ is isometric to the linking form
represented by
\begin{equation}\label{eq:asB}
B=\begin{pmatrix} d & 1\\ 1 & -k\end{pmatrix},
\end{equation}
where $d=\alpha(-1)\in\Z$ is such that $-(dk+1)$ is the (signed) determinant of $K$.
\end{corollary}
As in \cite[Section 5.2]{BF1} we can use Corollary~\ref{cor:linking_form_obstruction} to obstruct untwisting number $2$.

From Corollary~\ref{cor:linking_form_obstruction} we immediately recover Theorem~\ref{thm:iscyclic} from the introduction.
\begin{proposition}\label{prop:is_cyclic}
If $K$ is algebraically $k$--simple and $\Sigma(K)$ is the double branched cover, then $H_1(\Sigma(K);\Z)$ is cyclic.
\end{proposition}
\begin{remark}
It follows that Wendt's criterion for the unknotting number \cite{We37} coming from the double branched covers,
does not distinguish between knots that have unknotting number $1$ and knots that
are algebraically $k$--simple for some $k$.
\end{remark}
\begin{proof}[Proof of Proposition~\ref{prop:is_cyclic}]
By Corollary~\ref{cor:linking_form_obstruction} we infer that $H_1(\Sigma(K);\Z)\cong \Z^2/BZ^2$, where $B$ is as
in \eqref{eq:asB}. Subtract from the first column of $B$ the second column multiplied by $d$ to obtain the matrix
$\left(\begin{smallmatrix} 0 & 1 \\ 1+dk & -k\end{smallmatrix}\right)$. Then add to the second row the first one multiplied by $k$. We obtain the matrix
\[B'=\begin{pmatrix} 0 & 1 \\ 1+dk & 0\end{pmatrix}.\]
Row and column operations on matrices do not affect the cokernel, hence $\Z^2/B'\Z^2\cong\Z^2/B\Z^2$. Evidently we have $\Z^2/B'\Z\cong\Z/|dk+1|\Z$.
\end{proof}


\begin{thebibliography}{999}
\bibitem{Bl57}
R. C. Blanchfield, {\em  Intersection theory of manifolds with operators with applications to
knot theory}, Ann. of Math. 65 (1957), 340--356.
\bibitem{BF3} M.~Borodzik, S.~Friedl, \emph{On the algebraic unknotting number}
Trans. London Math. Soc. \textbf{1} (2014), no. 1, 57--84.

\bibitem{BF2} M.~Borodzik, S.~Friedl, \emph{Unknotting number and classical invariants. II.},
Glasg. Math. J. \textbf{56} (2014), no. 3, 657--680.

\bibitem{BF1} M.~Borodzik, S.~Friedl, \emph{Unknotting number and classical invariants. I.},
Alg. Geom. Top. \textbf{15} (2015), no. 1, 85--135.

\bibitem{Con} A.~Conway, \emph{An explicit computation of the Blanchfield pairing for arbitrary links},
preprint, arxiv:1706.00226.

\bibitem{CFT} A.~Conway, S.~Friedl, E.~Toffoli, \emph{The Blanchfield pairing of colored links},
preprint, arxiv:1609.08057.


\bibitem{Fo93}
M. Fogel, {\em The Algebraic Unknotting Number}, PhD thesis, University of California, Berkeley (1993).

\bibitem{COT03}
T. Cochran, K. Orr and P. Teichner, {\em
Knot concordance, Whitney towers
and $L^2$-signatures}, Annals of Mathematics, \textbf{157} (2003), 433--519.
\bibitem{Fo94}
M. Fogel, {\em Knots with Algebraic Unknotting Number One}, Pacific J. Math. \textbf{163} (1994), 277--295.
\bibitem{FP} S.~Friedl, M.~Powell, \emph{A calculation of Blanchfield pairings of 3-manifolds and knots},
preprint, arxiv 1512.04603.

\bibitem{FreedmanQuinn}
M. Freedman and F. Quinn, {\em Topology of 4--manifolds}, Princeton Mathematical
Series 39, Princeton University Press, Princeton, NJ (1990).

\bibitem{Stipsicz} Robert~Gompf and Andr\'as~Stipsicz,
\textit{4-manifolds and Kirby calculus},
Graduate Studies in Mathematics. \textbf{20}. Providence, RI. 1999.

\bibitem{Hi12}
J. Hillman, {\em  Algebraic invariants of links}, second edition, Series on Knots and Everything, 52. World Scientific Publishing Co. Inc., River Edge, NJ, (2012).
\bibitem{Kear}
C.  Kearton, {\em
Blanchfield duality and simple knots},
Trans. Am. Math. Soc. \textbf{202} (1975), 141--160.

\bibitem{Muk90}
H. Murakami, {\em Algebraic Unknotting Operation}, Q\&A. Gen. Topology \textbf{8} (1990), 283--292.
\bibitem{Przy} J.~Przytycki, \emph{$t_k$ moves on links}, Braids (Santa Cruz, CA, 1986), 615--656,
Contemp. Math., 78, Amer. Math. Soc., Providence, RI, 1988.

\bibitem{Ran81}
A. Ranicki, {\em Exact sequences in the algebraic theory of surgery}, Mathematical
Notes 26, Princeton (1981).
\bibitem{Ran02}
A. Ranicki, {\em Algebraic and geometric surgery},
Oxford Mathematical Monographs. Oxford Science Publications. The Clarendon Press, Oxford University Press, Oxford, 2002.

\bibitem{Rol75}
D.~Rolfsen,
\emph{A surgical view of Alexander's polynomial}, Geometric topology (Proc. Conf., Park City, Utah, 1974), pp. 415--423. 
Lecture Notes in Math., Vol. 438, Springer, Berlin, 1975. 

\bibitem{Sae99}
O. Saeki, {\em On Algebraic Unknotting Numbers of Knots}, Tokyo J. Math. \textbf{22} (1999), 425-443.
\bibitem{Sav}
N. Saveliev,  \emph{Lectures on the topology of 3-manifolds. An introduction to the Casson invariant},
De Gruyter Textbook. Walter de Gruyter \& Co., Berlin, 1999.
\bibitem{We37}
H. Wendt, {\em Die gordische Aufl\"osung von Knoten}, Math. Z.  \textbf{42} (1937), 680--696.
\end{thebibliography}
\end{document}